\documentclass[11pt]{article}
\usepackage{amsmath,amsthm,amssymb}
\usepackage[all]{xy}

\usepackage{color}
\definecolor{citation}{rgb}{0,.40,.80}
\usepackage[colorlinks,linkcolor=citation,citecolor=citation,urlcolor=citation]{hyperref}

\title{Rational points and derived equivalence}

\author{Nicolas Addington,
Benjamin Antieau, \\
Sarah Frei, and
Katrina Honigs}

\date{}

\newcommand \C {\mathbb C}
\newcommand \cC {\mathcal C}
\newcommand \F {\mathbb F}
\newcommand \G {\mathbb G}
\renewcommand \L {\mathcal L}
\newcommand \M {\mathcal M}
\renewcommand \O {\mathcal O}
\renewcommand \P {\mathbb P}
\newcommand \Q {\mathbb Q}
\newcommand \R {\mathbb R}
\newcommand \W {\mathbb W}
\newcommand \X {\mathcal X}
\newcommand \Z {\mathbb Z}

\DeclareMathOperator \Br {Br}
\DeclareMathOperator \Pic {Pic}
\DeclareMathOperator \NS {NS}
\newcommand \Picbar {\overline{\mathrm{Pic}}{}} 
\DeclareMathOperator \Spec {Spec}
\DeclareMathOperator \Ext {Ext}

\newtheorem {thm} {Theorem}
\newtheorem {lem} [thm] {Lemma}
\newtheorem {prop} [thm] {Proposition}
\theoremstyle{definition}
\newtheorem* {question*} {Question}
\newtheorem {defprop} [thm] {Definition--Proposition}

\begin{document}

\maketitle

\begin{abstract}
\noindent
We give the first examples of derived equivalences between varieties defined over non-closed fields where one has a rational point and the other does not.  We begin with torsors over Jacobians of curves over $\Q$ and $\F_q(t)$, and conclude with a pair of hyperk\"ahler 4-folds over $\Q$.  The latter is independently interesting as a new example of a transcendental Brauer--Manin obstruction to the Hasse principle.
\end{abstract}

\section{Introduction}

The following question was posed by Esnault and stated in \cite[Question~2]{ab}:

\begin{question*}
For smooth projective varieties defined over a non--algebraically closed field $k$, is the existence of a $k$-rational point preserved under derived equivalence?
\end{question*}

Honigs et al.\ \cite{honigs1,honigs2} showed that over a finite field, the answer is yes up to dimension 3, and for Abelian varieties of any dimension; in fact the \emph{number} of rational points is preserved in these cases.  A conjecture of Orlov \cite[Conj.~1]{orlov_motives} would imply that this continues to hold in higher dimensions.

Antieau, Krashen, and Ward \cite[Thm.~1.1]{akw} showed that the answer is yes for curves of genus 1 over arbitrary fields.  (For curves of genus $g \ge 2$ or $g = 0$, the canonical bundle is ample or anti-ample, so there are no interesting derived equivalences.)

Hassett and Tschinkel \cite{ht_rational} studied the question for K3 surfaces, giving a positive answer over $\R$ and proving a number of suggestive results over local fields.  They showed that the \emph{index} of a K3 surface -- that is, the greatest common divisor of the degrees of its closed points -- is invariant under derived equivalence.  Thus if $D^b(X) \cong D^b(Y)$ and $X$ has a rational point then $Y$ has a zero-cycle of degree 1, so it seems very difficult to find a counterexample among K3 surfaces.  On the other hand, they asked whether the existence of a rational point is preserved under \emph{twisted} derived equivalence, and Ascher, Dasaratha, Perry, and Zhou \cite{adpz} showed that it is not, producing pairs of K3 surfaces over $\Q$, $\Q_2$, and $\R$ such that $D^b(X,\alpha) \cong D^b(Y,\beta)$ for some $\alpha \in \Br(X)$ and $\beta \in \Br(Y)$, where $X$ has a rational point (at which the restriction of $\alpha$ is trivial) and $Y$ has no rational points.

Auel and Bernardara studied geometrically rational surfaces in \cite{ab}, showing (among other things) that a del Pezzo surface $S$ of degree at least 5 has a rational point if and only if $D^b(S)$ admits a full exceptional collection.
\bigskip

We give two classes of examples to show that the answer to the question above is no: first, some Abelian varieties and torsors over them, and second, a pair of hyperk\"ahler 4-folds.

\begin{thm} \label{first_thm}
For every $g \ge 2$, there is an Abelian $g$-fold $X$ defined over $\Q$, an $X$-torsor $Y$ with no rational points, and a $\Q$-linear exact equivalence
\[ D^b(X) \cong D^b(Y). \]
The same holds over $\F_q(t)$ for any odd $q$.
\end{thm}

\noindent Recall that a torsor over an Abelian variety has a rational point if and only if it has a zero-cycle of degree 1, so this theorem contrasts with the behavior of K3 surfaces discussed above.

To prove Theorem~\ref{first_thm}, we mine the literature for curves $C$ defined over $\Q$ and $\F_q(t)$ such that $\Pic^{g-1}_C$ has no rational points.  If $k = \R$, $\Q_p$, or $\F_q((t))$ then $\Pic^{g-1}_C$ always has $k$-points by work of Lichtenbaum \cite{lichtenbaum} and a remark of Poonen and Stoll \cite[Footnote 10]{poonen_stoll}, so for the examples we are interested in, $\Pic^{g-1}_C$ is a counterexample to the Hasse principle.  Then Theorem~\ref{first_thm} follows from the next result:

\begin{thm} \label{curve_thm}
If $C$ is a smooth, projective, geometrically connected curve of genus $g \ge 1$ over an arbitrary field $k$, then there is a $k$-linear exact equivalence
\[ D^b(\Pic^0_C) \cong D^b(\Pic^{g-1}_C). \]
\end{thm}
\noindent The proof is a repackaging of Mukai's classic derived equivalence \cite{fourier_mukai} between an Abelian variety and its dual.  We include a broader discussion of torsors over Abelian varieties in \S\ref{ab_sec}.
\bigskip

The varieties $X$ and $Y$ in Theorem~\ref{first_thm} become isomorphic after a finite field extension, but in \S\ref{hk_sec} we present a more sophisticated counterexample over $\Q$, in which the varieties remain different even over $\C$.  These are our hyperk\"ahler 4-folds.

\begin{thm} \label{hk_thm}
There is an explicit K3 surface $S$, defined over $\Q$, and two smooth, projective, 4-dimensional moduli spaces $X$ and $Y$ of sheaves on $S$, such that $X$ has infinitely many rational points, $Y$ has no zero-cycle of degree 1, and there is a $\Q$-linear exact equivalence
\[ D^b(X) \cong D^b(Y). \]
The spaces $X$ and $Y$ are not birational, even over $\C$.
\end{thm}

In fact $Y$ has points over $\R$ and over $\Q_p$ for every prime $p$, hence is a counterexample to the Hasse principle.  We use the class $\alpha \in \Br(Y)$ that obstructs the existence of a universal sheaf on $S \times Y$ as a Brauer--Manin obstruction.  It is a transcendental Brauer class: it remains non-trivial in $\Br(Y_{\overline\Q})$ or $\Br(Y_\C)$.

This counterexample is related to the Abelian counterexamples above in that $X$ is fibered over $\P^2$ in Jacobians of curves of genus 2, and $Y$ is fibered in $\Pic^1$ of the same curves.  The derived equivalence is a version of our earlier equivalence $D^b(\Pic^0_C) \cong D^b(\Pic^1_C)$ in families; the extension to the singular fibers is due to Arinkin \cite{arinkin}, and was used in work of Addington, Donovan, and Meachan \cite{adm}, but we give a simplified description.

By taking fibers of $X$ and $Y$ over $\Q$-points of $\P^2$ we can get more explicit examples of Theorem~\ref{first_thm}.  By taking generic fibers we get an example of Theorem~\ref{first_thm} over the function field $\Q(x,y)$, or indeed $\C(x,y)$.  By taking the preimage of a general curve in $\P^2$ we get derived equivalent 3-folds of Kodaira dimension 1, where one has a rational point and one does not.

Theorem~\ref{hk_thm} stands in contrast to a result of Frei \cite[Thm.~1]{frei}, who showed that over a \emph{finite} field, two smooth projective moduli spaces of sheaves on a given K3 surface have the same number of points as soon as they have the same dimension.

\paragraph{Acknowledgments}
The computations in \S\ref{hk_sec} were done in Magma \cite{magma}, and drew heavily on the program accompanying Berg and V\'arilly-Alvarado's paper \cite{bva}; we thank them for expert advice, S.~Elsenhans for a crucial Magma tip, and B.~Young for computer time.  We also thank A.~Auel for suggesting Lemma~\ref{integral_reduction}, which simplified our arguments,
J.-L.~Colliot-Th\'el\`ene, 
E.~Elmanto, 
R.~Takahashi, 
S.~Tirabassi,
B.~Viray, and 
D.~Zureick-Brown 
for helpful discussions, and the referees for their comments and corrections.

B.A.\ was supported by NSF grant no.\ DMS-1552766, and by grant no.\ DMS-1440140 while in residence at the Mathematical Sciences Research Institute in Spring 2019.  K.H.\ was supported by NSF grant no.\ DMS-1606268.

\paragraph{Conventions}
For an arbitrary field $k$, we let $k^s$ denote its separable closure, $\bar k$ its algebraic closure, and $G = \operatorname{Gal}(k^s/k)$ its absolute Galois group.  For a variety $X$ over $k$, we let $X^s = X \times_k k^s$ and $\bar X = X \times_k \bar k$.  We let $\Pic_X$ denote the Picard scheme and $\Pic(X)$ the Picard group; the group of $k$-rational points of $\Pic_X$ is $\Pic(X^s)^G$, which may be strictly bigger than $\Pic(X)$, as we will discuss at length.

Because we consider sheaves of rank 0 and sheaves on reducible spaces, stability for us always means Gieseker stability, defined in terms of the reduced Hilbert polynomial, rather than slope stability.

\section{Abelian counterexamples} \label{ab_sec}

\subsection{Torsors over Abelian varieties}
Let $A$ be an Abelian variety over an arbitrary field $k$.  As Poonen and Stoll discuss in \cite[\S4]{poonen_stoll}, the short exact sequence of $G$-modules
\[ 0 \to \Pic^0(A^s) \to \Pic(A^s) \to \NS(A^s) \to 0 \]
yields a long exact sequence
\begin{multline*}
0 \to \Pic^0(A) \to \Pic(A) \to \NS(A^s)^G \\
\to H^1(G, \Pic^0(A^s)) \to H^1(G, \Pic(A^s)) \dotsb.
\end{multline*}
For any divisor class
\[ \lambda \in \NS(A^s)^G, \]
we can consider the associated component of the Picard scheme
\[ \Pic^\lambda_A \subset \Pic_A, \]
which is a torsor over the dual Abelian variety $\hat A = \Pic^0_A$, trivial if and only if $\lambda$ is in the image of the map $\Pic(A) \to \NS(A^s)^G$ above: that is, in $\NS(A)$.
\begin{thm} \label{ab_thm}
There is a $k$-linear exact equivalence
\[ D^b(A) \cong D^b(\Pic^\lambda_A). \]
\end{thm}

\noindent It is interesting to contrast this statement with \cite[Cor.~5.2]{akw}, which says that in characteristic zero, if $\operatorname{End}(A^s) = \Z$ and the inclusion $\NS(A) \subset \NS(A^s)^G = \Z$ is an equality, then two derived equivalent $A$-torsors necessarily generate the same subgroup of $H^1(G, \Pic^0(A^s))$.

\begin{proof}[Proof of Theorem~\ref*{ab_thm}]
Because $A$ has a rational point, there is a Poincar\'e line bundle $P$ on $A \times \Pic^\lambda_A$ by \cite[Exercise 9.4.3]{kleiman} or by Lemma~\ref{point_kills_alpha} below.  We claim that the functor
\[ F\colon D^b(A) \to D^b(\Pic^\lambda_A) \]
induced by $P$ is an equivalence.  By \cite[Lem.~2.12]{orlov_abelian}, it is enough to prove it after base change to $\bar k$.

So assume that $k$ is algebraically closed, choose a base point $[M] \in \Pic^\lambda(A)$, and use it to identify $\Pic^\lambda_A$ with $\hat A$.  Then $P \otimes \pi_1^* M^\vee$ is a Poincar\'e bundle on $A \times \hat A$, hence induces an equivalence $D^b(A) \to D^b(\hat A)$ by Mukai's theorem \cite{fourier_mukai}; see also Polishchuk's book \cite[Ch.~11 and 17]{polishchuk_ab}, or Huybrechts' book \cite[Ch.~9]{huybrechts_fm} for another account in characteristic zero.  Now the functor induced by $P \otimes \pi_1^* M^\vee$ is just tensoring with $M^\vee$ (which is an equivalence) followed by $F$, so it is an equivalence if and only if $F$ is.
\end{proof}

\subsection{Proof of Theorem~\ref*{curve_thm} and discussion of Jacobians} \label{jacobian_sec}

We could deduce Theorem~\ref{curve_thm} from Theorem~\ref{ab_thm} by letting $A = \Pic^0_C$ and letting $\lambda \in \NS(A^s)^G$ be the class of the $\Theta$-divisor; then $\Pic^\lambda_A \cong \Pic^{g-1}_C$ by \cite[Cor.~4]{poonen_stoll}.\footnote{Recall that the $\Theta$-divisor lives most naturally in $\Pic^{g-1}_C$, as the image of the Abel--Jacobi map $\operatorname{Sym}^{g-1}(C) \to \Pic^{g-1}_C$.  If $\Pic^{g-1}_C$ has a $k$-point then we can identify $\Pic^0_C \cong \Pic^{g-1}_C$ and get a $\Theta$-divisor on $A = \Pic^0_C$, unique up to translation.  But if $\Pic^{g-1}_C(k) = \varnothing$ then we only get a class in $\NS(A^s)^G$, not $\NS(A)$.}

But we can also describe the Poincar\'e bundle on $\Pic^0_C \times \Pic^{g-1}_C$ as the line bundle associated to a very explicit divisor $D$, and this will prove useful in \S\ref{hk_sec}.  Over $\bar k$, we can write
\[ D = \Bigl\{ (L,M) : H^1(L \otimes M) \ne 0 \Bigr\} \subset \Pic^0_{\bar C} \times \Pic^{g-1}_{\bar C}. \]
We see that the fiber of $D$ over a point of $\Pic^{g-1}_{\bar C}$ is a translate of the $\Theta$-divisor on $\Pic^0_{\bar C}$, in such a way that $\O(D)$ is a universal bundle for $\Pic^{g-1}_{\bar C}$ as a moduli space of line bundles on $\Pic^0_{\bar C}$.

To see that $D$ is defined over $k$, we can describe it as the support of a certain twisted sheaf.  Let $\alpha \in \Br(\Pic^0_C)$ be the obstruction to the existence of a universal line bundle on $C \times \Pic^0_C$, and let $\L$ be the $\pi_2^* \alpha$-twisted universal bundle.  Similarly, let $\beta \in \Br(\Pic^{g-1}_C)$ and $\M$ on $C \times \Pic^{g-1}_C$.  Then $D$ is the support of the $(\alpha \boxtimes \beta)$-twisted sheaf
\begin{equation} \label{rank_0_sheaf}
R^1 \pi_{23,*}(\pi_{12}^* \L \otimes \pi_{13}^* \M),
\end{equation}
where $\pi_{ij}$ are the projections from $C \times \Pic^0_C \times \Pic^{g-1}_C$ onto any two factors.
\pagebreak 

For another argument that $\O(D)$ agrees with the Poincar\'e line bundle, we could use Deligne's description of the latter as the line bundle whose fiber at a geometric point $(L,M)$ is
\begin{equation} \label{deligne}
(\det H^*(L \otimes M))^{-1} \otimes \det H^*(L) \otimes \det H^*(M) \otimes (\det H^*(\O_C))^{-1}.
\end{equation}
We learned this description from Arinkin's paper \cite{arinkin}; see Polishchuk's book \cite[\S22.3]{polishchuk_ab} for another account.  Globally, the first term of \eqref{deligne} is
\[ (\det R^0 \pi_{23,*}(\pi_{12}^* \L \otimes \pi_{13}^* \M))^{-1} \otimes (\det R^1 \pi_{23,*}(\pi_{12}^* \L \otimes \pi_{13}^* \M)). \]
Now $R^0 \pi_{23,*}(\pi_{12}^*\L\otimes\pi_{13}^*\M)$ vanishes, and $R^1
\pi_{23,*}(\pi_{12}^*\L\otimes\pi_{13}^*\M)$ is supported on $D$ and has generic rank 1 there, so its determinant as a sheaf on $\Pic^0_C \times \Pic^{g-1}_C$ is the (untwisted) line bundle $\O(D)$.  The other terms of \eqref{deligne} just modify $\O(D)$ by line bundles pulled back from $\Pic^0_C$ or $\Pic^{g-1}_C$, so they don't change the fact that it induces a derived equivalence.

The interested reader may consult \cite[Proof of Prop.~2.1(a)]{adm} for more details on the generalization of Mukai's equivalence from $\Pic^0_C \times \Pic^0_C$ to $\Pic^m_C \times \Pic^n_C$, and [ibid., Rmk.~2.3] for more (twisted) derived equivalences.

\subsection{Proof of Theorem~\ref*{first_thm}}

To deduce Theorem~\ref{first_thm} from Theorem~\ref{curve_thm}, it is enough to find genus-$g$ curves $C$ such that $\Pic^{g-1}_C$ has no rational points.  There is a subtlety, in that the inclusion
\begin{equation} \label{pic_inclusion_for_curves}
\Pic^d(C) \subset \Pic^d_C(k) = \Pic^d(C^s)^G
\end{equation}
may be proper in general.

Coray and Manoil \cite[Prop.~4.2]{coray_manoil} showed that for any $g \ge 1$, the hyperelliptic curve $C$ of genus $g$ determined by
\[ y^2 = 605 \cdot 10^6 x^{2g+2} + (18x^2-4400)(45x^2-8800) \]
has points over $\R$ and over $\Q_p$ for every prime $p$, so the inclusion \eqref{pic_inclusion_for_curves} is an equality for all $d$ by [ibid., Cor.~2.5]; but that $\Pic^1(C) = \varnothing$.  Bhargava, Gross, and Wang \cite[Thm.~2]{bgw} later showed that a positive proportion of hyperelliptic curves over $\Q$ have this property, by studying pencils of quadrics.  Note that if $C$ is hyperelliptic and $g$ is even then $\Pic^1_C \cong \Pic^{g-1}_C$.

Poonen and Stoll gave explicit hyperpelliptic curves of even genus over $\Q$ for which $\Pic^{g-1}_C$ has no rational points in \cite[Props.~26, 27, 28]{poonen_stoll}, and studied their density in [ibid.,~\S9].  They gave a non-hyperelliptic example of genus 3 over $\Q$, namely the plane quartic curve
\[ x^4 + py^4 + p^2 z^4 = 0 \]
for any $p \equiv -1 \pmod{16}$, in [ibid., Prop.~29].  And they gave a genus-2 example over $\F_q(t)$ for $q$ odd, namely
\[ y^2 = tx^6 + x - at \]
for any $a$ that is not a square in $\F_q$, in [ibid., Prop.~30].

While we expect that there are similar examples over $\Q$ in odd genus $g \ge 5$, and over $\F_q(t)$ in genus $g \ge 3$, we did not find them in the literature.  But to finish the proof of Theorem~\ref{first_thm}, we can just take the examples above and cross with an elliptic curve: let $C$ be a genus-2 curve such that $Y = \Pic^1_C$ is a non-trivial torsor over the Abelian surface $X = \Pic^0_C$, and let $E$ be an elliptic curve; then for any $g > 2$ we see that $Y \times E^{g-2}$ is a non-trivial torsor over $X \times E^{g-2}$, and the two have equivalent derived categories \cite[Exercise~5.20]{huybrechts_fm}.  \qed

\section{Brauer classes on compactified Picard schemes} \label{big_pic_sec}

Let $X$ be a projective, geometrically integral, but not necessarily smooth variety over an arbitrary field $k$.  For a divisor class $\lambda \in \NS(X^s)^G$, there is an inclusion
\begin{equation} \label{pic_inclusion}
 \Pic^\lambda(X) \ \subset \ \Pic^\lambda_X(k) = \Pic^\lambda(X^s)^G,
\end{equation}
which we encountered for smooth curves in the proof of Theorem~\ref{first_thm} but which we now consider more broadly in preparation for \S\ref{hk_sec}.

Coray and Manoil \cite{coray_manoil} have studied the inclusion \eqref{pic_inclusion} using the exact sequence
\begin{equation} \label{hochschild-serre_thing}
0 \to \Pic(X) \to \Pic(X^s)^G \to \Br(k) \to \Br(X) \to \dots
\end{equation}
coming from the Hochschild--Serre spectral sequence.  We will study it using a class in $\Br(\Pic^\lambda_X)$ that arises from viewing the Picard scheme as a moduli space of stable sheaves on $X$.

For standard results on moduli spaces of sheaves, our reference is Huybrechts and Lehn's book \cite{huybrechts_lehn}, together with Langer's papers \cite{langer1, langer2} which make the results available in positive and mixed characteristic.  The key point for us is that a moduli space $M$ of (S-equivalence classes of) semi-stable sheaves with given Hilbert polynomial is projective, and the open subspace $M_\text{stab}$ 
parametrizing geometrically stable sheaves is in general only quasi-projective but carries a Brauer class that obstructs the existence of a universal sheaf on $X \times M_\text{stab}$.  This Brauer class is associated to a principal PGL$_n$-bundle coming from the GIT construction of the moduli space \cite[Cor.~4.3.5]{huybrechts_lehn}.  It can be realized as an \'etale $\P^n$-bundle as follows: by boundedness, there is an $N \gg 0$, depending only on the Hilbert polynomial, such that for every semi-stable sheaf $F$, the twist $F(N)$ is globally generated and has no higher cohomology; then we take the bundle over $M_\text{stab}$ whose fiber at $F$ is $\P H^0(F(N))$.  
For a stackier description, we could say that a geometrically stable sheaf is simple, so its automorphism group is $\G_m$, so the moduli stack of geometrically stable sheaves is a $\G_m$-gerbe over the moduli space; compare \cite[Cor.~4.3.3]{Lieblich}.

Because $X$ is geometrically integral, every rank-1 torsion-free sheaf on $X$ is geometrically stable with respect to any ample line bundle.  If $X$ is smooth, then $\Pic^\lambda_X$ is closed in the moduli space of rank-1 torsion-free sheaves on $X$, but if not then we can consider its closure
\[ \Picbar^\lambda_X. \]
Let
\[ \alpha_\lambda \in \Br(\Picbar^\lambda_X) \]
be the Brauer class that obstructs the existence of a universal sheaf on $X \times \Picbar^\lambda_X$.\footnote{Note that $\alpha_\lambda$ is the restriction of a natural Brauer class that lives on the whole moduli space of geometrically stable sheaves, or indeed of simple sheaves, so there is no need to worry about ramification at the boundary components of $\Picbar^\lambda_X$.}

The inclusion \eqref{pic_inclusion} above can be understood as follows: for any extension field $K/k$, we have
\[ \Pic^\lambda(X_K) = \{ \ell \in \Pic^\lambda_X(K) : \alpha_\lambda|_\ell = 0 \in \Br(K) \}. \]

The following lemma is essentially well-known for $\Pic^\lambda_X$ -- see \cite[Cor.~2.3]{coray_manoil} for an approach using the exact sequence \eqref{hochschild-serre_thing} -- but in the next section we need it for $\Picbar^\lambda_X$, which requires different methods.

\begin{lem} \label{point_kills_alpha}
Let $X$ be a projective, geometrically integral variety over an arbitrary field $k$, let $\lambda \in \NS(X^s)^G$, and let $\Picbar^\lambda_X$ and $\alpha_\lambda \in \Br(\Picbar^\lambda_X)$ be as above.  If $X$ has a smooth $k$-point, or more generally a zero-cycle of degree 1 supported in its smooth locus, then $\alpha_\lambda = 0$.
\end{lem}
\begin{proof}
Let $c$ be the class in $K_\text{num}(X^s)^G$ corresponding to $\lambda \in \NS(X^s)^G$.  First we argue that for any vector bundle $E$ on $X$, defined over $k$, the number $N_E := \chi(c \cdot [E])$ satisfies $N_E \cdot \alpha_\lambda = 0$.  A similar claim is valid for any moduli space of geometrically stable sheaves, not only $\Picbar^\lambda_X$.  Let $\pi_1$ and $\pi_2$ be the projections from $X \times \Picbar^\lambda_X$, and let $U_\lambda$ be the $\pi_2^* \alpha_\lambda$-twisted universal sheaf on $X \times \Picbar^\lambda_X$.  Then $R\pi_{2,*}(U_\lambda \otimes \pi_1^* E)$ is an $\alpha_\lambda$-twisted perfect\footnote{Because $U$ is flat over $\Picbar^\lambda_X$ and $X$ is proper; see \cite[\href{https://stacks.math.columbia.edu/tag/0DJQ}{Tag 0A1E}]{stacks-project}.} complex of rank $N_E$, so $N_E\cdot\alpha_\lambda = 0$.  (See~\cite[Thm~4.6.5]{huybrechts_lehn} for a similar argument.)

Now if $\xi \subset X_\text{sm}$ is a subscheme of length $d$, then $\O_\xi$ admits a finite resolution by vector bundles and satisfies $\chi(c \cdot [\O_\xi]) = d$.  Thus a zero-cycle of degree 1 forces $\alpha_\lambda$ to vanish, by a g.c.d.\ argument.  Note that if the support of $\xi$ meets the singular locus of $X$ then $\O_\xi$ might not admit a finite resolution by vector bundles.
\end{proof}

Lemma~\ref{point_kills_alpha} will be used in conjunction with the following:

\begin{lem} \label{integral_reduction}
Let $X$ be a projective variety over a non-Archimedean local field $k = k_v$ with ring of integers $\O_v$ and residue field $\kappa_v$.  Suppose that $X$ has a model $\X \to \Spec \O_v$ whose special fiber is geometrically integral, or more generally has a component $Y$ of multiplicity 1 that is geometrically integral.  Then $X$ has a zero-cycle of degree 1 supported in its smooth locus.
\end{lem}
\begin{proof}
By the Lang--Weil bounds, $Y$ has a smooth point defined over a degree-$r$ extension of $\kappa_v$ for all $r \gg 0$ \cite[Thm.~7.7.1]{poonen_book}.  By Hensel's lemma, this lifts to a smooth point of $X$ defined over a degree-$r$ extension of $k_v$, so there is a zero-cycle of degree $r$ defined over $k_v$, supported in the smooth locus of $X$.  By considering $r$ and $r+1$ we get a zero-cycle of degree 1.
\end{proof}

\section{A hyperk\"ahler counterexample} \label{hk_sec}

Recall that a K3 surface $S$ of degree 2 can be obtained as a double cover of $\P^2$ branched over a smooth sextic curve.  In \S\ref{steps} we consider two moduli spaces $X$ and $Y$ of sheaves on such a K3 surface $S$ and show that if the sextic satsifies a laundry list of conditions then the conclusions of Theorem~\ref{hk_thm} hold.  In \S\ref{the_example} we exhibit a sextic satsifying those conditions.  In \S\ref{computer_chat} we discuss computational issues that we faced in finding this sextic.

\subsection{Steps of the proof} \label{steps}

Let $f \in \Z[x,y,z]$ be a homogeneous polynomial of degree 6 that cuts out a smooth curve $B \subset \P^2_\Q$.  Let $S$ be the K3 surface over $\Q$ defined by
\[ w^2 = f(x,y,z), \]
either in weighted projective space $\W\P^3(3,1,1,1)$ or in the total space of $\O_{\P^2}(3)$.  Let $\pi\colon S \to \P^2$ be the map that forgets $w$, which is a double cover branched over $B$.

The class $h := \pi^*\O_{\P^2}(1) \in \NS(S)$ is ample
and satisfies $h^2 = 2$.
A curve in the linear system $|h|$ is the preimage of a line in $\P^2$,
and we may identify $|h|$ with the space of lines in $\P^2$.  A general member of $|h|$ is a smooth curve
of genus 2.

\begin{defprop} \label{tritangent}
A line $L \subset \P^2$ (defined over any perfect field $k$) is a \emph{tritangent line} to the sextic curve $B$ if it satisfies one of the following conditions, which are equivalent:
\begin{enumerate}
\item $L$ is tangent to $B$ at three points, in the scheme-theoretic sense;
\item $f|_L = c \cdot g^2$, for some scalar $c$ and some cubic $g \in H^0(\O_L(3))$;
\item the curve $C = \pi^{-1}(L)$ is not geometrically integral.
\end{enumerate}
\end{defprop}
\begin{proof}
It is clear that (a) $\Leftrightarrow$ (b) $\Rightarrow$ (c).  To see that (c) $\Rightarrow$ (b), suppose for simplicity that the line is given by $z=0$.  Then the curve $w^2 = f(x,y)$ in $\W\P^3(3,1,1)$ is not integral over $\bar k$, so $w^2 - f(x,y)$ factors, so $f(x,y)$ is a square in $\bar k[x,y]$.  Because $k$ is perfect, this implies that $f = c \cdot g^2$ for some $c \in k$ and some cubic $g \in k[x,y]$. 
\end{proof}

Consider the moduli spaces of semi-stable sheaves on $S$ of rank 0, first Chern class $h$, and Euler characteristic $-1$ or 0:
\begin{align*}
X &:= M_h(0,h,-1), \\
Y &:= M_h(0,h,0).
\end{align*}
Sheaves with these invariants include all line bundles of degree 0 or 1 supported on curves in the linear system $|h|$, together with some rank-1 torsion-free sheaves supported on singular curves.  The moduli spaces are projective varieties of dimension 4.  They map to $|h|$ by sending a sheaf to its support.  Thus if $\cC \to |h|$ is the
tautological family of curves in the linear system $|h|$, that is,
\[ \cC = \{ (x,C) : x \in C \} \subset S \times |h|, \]
then $X$ is a compactification of the relative Picard scheme $\Pic^0_{\cC/|h|}$, and $Y$ is a compactification of $\Pic^1_{\cC/|h|}$.

\pagebreak 
\begin{prop} \label{Y_is_smooth}
The space $X$ parametrizes only geometrically stable sheaves, hence is smooth.  The same is true of $Y$ if there are no tritangent lines to $B$ defined over $\overline\Q$.
\end{prop}
\begin{proof}
Mukai showed that a moduli space of geometrically stable sheaves on a K3 surface is smooth \cite[Thm.~0.1]{mukai_inventiones}.

Let $F$ be a pure sheaf on $S_{\overline\Q}$ with $\operatorname{rank}(F) = 0$ and $c_1(F) = h$.  Write $c_1(F) = m_1 [C_1] + \dotsb + m_k [C_k]$, where $C_i$ are the irreducible components of the reduced support of $F$ and $m_i > 0$.  Since $c_1(F).h = h^2 = 2$ and $h$ is ample, we see that either $F$ is supported on an irreducible curve $C \in |h|$ and has generic rank 1 there, or it is supported on a reducible curve $C_1 \cup C_2$ with $C_1.h = C_2.h = 1$, and has generic rank 1 on each component.

In the first case (irreducible support), any saturated\footnote{See~\cite[Prop~.1.2.6]{huybrechts_lehn} for why it is enough to consider saturated subsheaves.} subsheaf $G \subset F$ is either 0 or all of $F$, so $F$ is necessarily stable.  In the second case, which can only occur if there is a tritangent line defined over $\overline\Q$, a saturated subsheaf $G \subset F$ might be supported on $C_1$ or $C_2$ alone, with generic rank 1 there.  In this case, the Hilbert polynomial $P_G(t)$ is $t + \chi(G)$, and the reduced Hilbert polynomial $p_G(t)$ is the same.  If $\chi(F) = -1$ then $P_F(t) = 2t - 1$, so $p_F(t) = t - \tfrac12$ so $p_G(t) \le p_F(t)$ implies $p_G(t) < p_F(t)$, so $F$ is again stable.
\end{proof}

\begin{prop}
Suppose there are no tritangent lines to $B$ defined over $\overline\Q$.  Then there is a $\Q$-linear exact equivalence $D^b(X) \cong D^b(Y)$.
\end{prop}
\begin{proof}
To emulate the construction of \S\ref{jacobian_sec} for the family of curves $\cC \to |h|$, we might want to define a divisor
\[ \{ (L,M) : H^1(L \otimes M) \ne 0 \} \subset X \times_{|h|} Y. \]
But on a singular curve $C \in |h|$, if $L$ and $M$ both fail to be locally free at some singular point, then the underived tensor product $L \otimes M$ is the wrong thing to write, and $\operatorname{Tor}_i^C(L,M)$ may be non-zero for infinitely many $i > 0$.

So following Arinkin \cite{arinkin}, we consider the open subschemes
\[ X^\circ = \Pic^0_{\cC/|h|} \subset X \qquad \text{and} \qquad Y^\circ = \Pic^1_{\cC/|h|} \subset Y \]
and the divisor
\[ D = \{ (L,M) : H^1(L \otimes M) \ne 0 \} \subset X^\circ \times_{|h|} Y \ \cup\ X \times_{|h|} Y^\circ, \]
which avoids the issue just discussed because one of $L$ and $M$ is always locally free.  To see that $D$ is defined over $\Q$, we can describe it as the support of the analogue of \eqref{rank_0_sheaf}.  Next, consider the inclusion
\[ j\colon X^\circ \times_{|h|} Y \ \cup\ X \times_{|h|} Y^\circ \hookrightarrow X \times_{|h|} Y. \]
Our equivalence will be induced by the sheaf $j_* \O(D)$.  Presumably this coincides with $\O(\bar D)$ -- or more precisely, the dual of the ideal sheaf of $\bar D$ -- where $\bar D$ is the closure of $D$ in $X \times_{|h|} Y$, but we will not need this. \bigskip

To prove that the functor $F\colon D^b(X) \to D^b(Y)$ induced by $j_* \O(D)$ is an equivalence, we can again base change to $\bar\Q$ by \cite[Lem.~2.12]{orlov_abelian}.  Then for closed points $x_1, x_2 \in X$ we consider the natural map
\begin{equation} \label{ext_map}
\Ext^*_X(\O_{x_1}, \O_{x_2}) \longrightarrow \Ext^*_Y(F(\O_{x_1}), F(\O_{x_2})).
\end{equation}
The skyscraper sheaves $\O_{x_i}$ are a spanning class \cite[Prop.~3.17]{huybrechts_fm}, so if \eqref{ext_map} is an isomorphism for all $x_1$ and $x_2$ then $F$ is fully faithful [ibid., Prop.~1.49], and hence is an equivalence because $\omega_X$ and $\omega_Y$ are trivial [ibid., Prop.~7.6].

Let $p\colon X \to |h|$ and $q\colon Y \to |h|$ be the Abelian fibrations discussed earlier, which map a sheaf to its support.  Because our kernel $j_* \O(D)$ is supported on the fiber product $X \times_{|h|} Y$, we see that $F(\O_{x_1})$ is supported on the fiber $q^{-1}(p(x_1)) \subset Y$, and similarly with $x_2$.  Thus if $x_1$ and $x_2$ lie in different fibers of $p$, then $F(x_1)$ and $F(x_2)$ have disjoint support, so both sides of \eqref{ext_map} are zero.

If $x_1$ and $x_2$ lie in the same fiber, choose an \'etale neighborhood $U \to |h|$ of $p(x_1) = p(x_2)$ over which the family of cures $\cC \to |h|$ admits a section, and use the section to identify $Y|_U$ with $X|_U$.  Now we use Arinkin's \cite[Thm.~C]{arinkin}, which applies to $\Picbar^0$ of any integral curve with planar singularities, and indeed to any family of such curves by Arinkin's remark (2) after his Theorem C.  Our curves have planar singularities because they are contained in a smooth surface $S$, and they are integral thanks to our hypothesis on tritangent lines.  As at the end of \S\ref{jacobian_sec}, our $\O(D)$ now coincides with Arinkin's Poincar\'e line bundle $P$ up to tensoring with a line bundle pulled back from $X|_U$ on either side, so our $j_* \O(D)$ coincides with his Poincar\'e sheaf $\bar P = j_* P$ [ibid., Lem.~6.1(2)], again up to line bundles on either side.  Thus $F(\O_{x_1})$ and $F(\O_{x_2})$ argee, up to a line bundle, with the images of $\O_{x_1}$ and $\O_{x_2}$ under the functor induced by $\bar P$, which is an equivalence; thus \eqref{ext_map} is an isomorphism.
\end{proof}
\pagebreak 

\begin{prop}
The space $X$ contains a copy of $\P^2$, hence has infinitely many $\Q$-points.
\end{prop}
\begin{proof}
The map $X \to |h|$ that sends a sheaf to its support has a section given by mapping a curve $C \in |h|$ to the trivial line bundle $\O_C$.
\end{proof}

\begin{prop} \label{violates_hasse}
The space $Y$ has points over $\R$ and over $\Q_p$ for every prime $p$.
\end{prop}
\begin{proof}
Let $C \in |h|$ be any smooth curve.  Then $\Pic^1_C \subset Y$ has both $\R$ and $\Q_p$-points by Lichtenbaum's result \cite{lichtenbaum} as discussed in the introduction.
\end{proof}

\begin{prop} \label{brauer_manin}
Suppose $f$ is chosen so that
\begin{enumerate}
\item $S(\R) = \varnothing$,
\item there are no tritangent lines to $B$ defined over $\overline \F_2$, and
\item for every tritangent line $L$ to $B$ defined over $\F_q$ with $q$ odd, the curve $C = \pi^{-1}(L) \subset S_{\F_q}$ consists of two reduced rational curves defined over $\F_q$ rather than $\F_{q^2}$.\footnote{That is to say, in Definition-Proposition~\ref{tritangent}(b) we can take $c=1$ and $g \ne 0$.}
\end{enumerate}
Then $Y(k) = \varnothing$ for every number field $k$ of odd degree over $\Q$.
\end{prop}
\begin{proof}
Because there are no tritangent lines over $\overline\F_2$, there are no tritangent lines defined over $\overline\Q$, so $Y$ parametrizes only geometrically stable sheaves by Proposition~\ref{Y_is_smooth}.  Thus there is a Brauer class $\alpha \in \Br(Y)$ that obstructs the existence of a universal sheaf on $S \times Y$.  We will use $\alpha$ as a Brauer--Manin obstruction to the existence of $k$-points on $Y$.

First we claim that for all $y \in Y(\R)$ we have $\alpha|_y \ne 0$.  If on the contrary $\alpha|_y = 0$, then $y$ represents a sheaf $F$ on $S_\R$.  Because $\chi(F) = 0$ and $c_1(F)^2 = h^2 = 2$, the degree of $c_2(F) \in \operatorname{CH}_0(S_\R)$ is 1 by Riemann--Roch.  But a zero-cycle defined over $\R$ is a linear combination of $\R$-points and $\C$-points, and we have $S(\R) = \varnothing$, so there can be no zero-cycle of odd degree.

Now fix a number field $k$ of odd degree over $\Q$.  We claim that for all non-Archimedean places $v$ of $k$ and all $y \in Y(k_v)$ we have $\alpha|_y = 0$.  Even stronger, we will show that for all curves $C \in |h|$ defined over $k_v$, the restriction of $\alpha$ to the fiber $\Picbar^1_C \subset Y_{k_v}$ is zero.  Let $\O_v$ be the ring of integers of $k_v$ and $\kappa_v \cong \F_q$ the residue field.  Because the sextic $f$ was defined over $\Z$, we get a model of $C$ over $\Spec \O_v$.  By hypothesis, the reduction $C_{\kappa_v}$ is either geometrically integral or is a reduced union of two rational curves; in either case, Lemmas~\ref{integral_reduction} and~\ref{point_kills_alpha} imply that the restriction of $\alpha$ to $\Picbar^1_C$ is zero.

To conclude, suppose that $y \in Y(k)$.  Because $[k:\Q]$ is odd, $k$ has an odd number of real places, so we see that $\sum_v \operatorname{inv}_v(\alpha|_y)$ is an odd multiple of $1/2$, which is impossible; see for example \cite[Prop.~8.2.2]{poonen_book}.
\end{proof}





\begin{prop} Suppose that $\Pic(S_{\bar\Q}) = \Z h$.
\begin{enumerate}
\item The schemes $X$ and $Y$ are not birational, even over $\C$.

\item If the Brauer class $\alpha \in \Br(Y)$ that obstructs the existence of a universal sheaf on $S \times Y$ obstructs the Hasse principle -- for example, if the hypotheses of Proposition~\ref{brauer_manin} are satisfied -- then $\alpha$ is transcendental, that is, its image in $\Br(Y_{\bar\Q})$ is non-zero.
\end{enumerate}
\end{prop}
\begin{proof}
Part (a) is due to Sawon \cite[Prop.~15]{sawon}: the discriminant of the Beauville--Bogomolov form on $\Pic(X_\C)$ is $-4$, while on $\Pic(Y_\C)$ it is $-1$.

For part (b), recall the filtration of the Brauer group
\[ \Br_0(Y) \subset \Br_1(Y) \subset \Br(Y), \]
where
\begin{align*}
\Br_0(Y) &:= \operatorname{im}(\Br(\Q) \to \Br(Y)), \text{ and} \\
\Br_1(Y) &:= \ker(\Br(Y) \to \Br(Y_{\bar\Q})).
\end{align*}
Classes in $\Br_0(Y)$ are called constant, those in $\Br_1(Y)$ are called algebraic, and the rest are called transcendental.  The Hochschild--Serre spectral sequence gives an isomorphism
\[ \Br_1(Y) / \Br_0(Y) \cong H^1(G, \Pic(Y_{\bar\Q})); \]
see for example \cite[Cor.~6.7.8 and Rmk.~6.7.10]{poonen_book}.

The usual identification of $\Pic(Y_{\bar\Q})$ with the orthogonal to the Mukai vector $(0,h,0)$ in the Mukai lattice $\Z \oplus \Pic(S_{\bar\Q}) \oplus \Z$ is valid as $G$-modules; to see this, note that all the maps appearing in \cite[Thm.~2.4(vi)]{charles} are $G$-equivariant, and see \cite[\S2]{frei} for the techniques needed to relax condition (C) of \cite[Def.~2.3]{charles}.

Because $\Pic(S_{\bar\Q}) = \Z h$, we see that $\Pic(Y_{\bar\Q}) \cong \Z^2$ with trivial $G$-action, so $H^1(G, \Pic(Y_{\bar\Q})) = 0$.  Thus if $\alpha$ were algebraic then it would be constant, but a constant class cannot obstruct the Hasse principle.
\end{proof}

\subsection{The explicit example} \label{the_example}

The sextic polynomial
\begin{multline*}
f(x,y,z) = -x^6 - x^5 z - x^4 y^2 - x^4 z^2 - x^3 y z^2 - x^2 y^2 z^2 \\
- x y^5 - x y^4 z - x z^5 - y^6 - y^3 z^3 - y^2 z^4 - y z^5 - z^6
\end{multline*}
satisfies all the hypotheses laid out in the previous section.  Magma code to verify the claims below is included in an ancillary file {\tt verify.magma}.  We also provide {\tt generate.magma} for readers who want to search for more examples, or adapt the code for their own purposes. \medskip

We find that $f(x,y,z) < 0$ for all $x,y,z \in \R^3 \setminus 0$, so $S(\R) = \varnothing$. \medskip

There are tritangent lines to $B$ defined over $\overline\F_p$ for five primes.  For $p=5$, the line $z = 4x+y$ is tritangent.  For $p=31$, the line $y = 24x + 23$ is tritangent.  For $p = 7517$, $84716037398136110308799$, and
\begin{align*}
&4424904772196959344085200612883251617292465803437757948\\
&5992572698404066491363246248977477562371729031497984350\\
&0902180031058767256453958545754450340721124283977338015\\
&3664612642260759001523868554216076825404419681,
\end{align*}
there are tritangent lines whose equations we omit.  In each case, there is a single tritangent line defined over $\F_p$, and its preimage in $S$ consists of two reduced rational curves defined over $\F_p$ rather than $\F_{p^2}$. \medskip

The curve $B$ is smooth over $\F_{31}$, hence over $\Q$.  Using the Magma routine {\tt WeilPolynomialOfDegree2K3Surface}, due to Elsenhans, we find the characteristic polynomial of Frobenius acting on $H^2_\text{\'et}(S_{\bar\F_{31}}, \Q_\ell(1))$:
\begin{multline*}
(t-1)^2(t^{20} - \tfrac{12}{31} t^{19} + \tfrac{15}{31} t^{18} - \tfrac{6}{31} t^{17} - \tfrac{3}{31} t^{16} - \tfrac{6}{31} t^{15} - \tfrac{5}{31} t^{14} \\
+ \tfrac{5}{31} t^{13} - \tfrac{13}{31} t^{12} + \tfrac{15}{31} t^{11} - \tfrac{22}{31} t^{10} + \tfrac{15}{31} t^9 - \tfrac{13}{31} t^8 + \tfrac{5}{31} t^7 \\
- \tfrac{5}{31} t^6 - \tfrac{6}{31} t^5 - \tfrac{3}{31} t^4 - \tfrac{6}{31} t^3 + \tfrac{15}{31} t^2 - \tfrac{12}{31} t + 1).
\end{multline*}
The degree-20 factor is irreducible and is not cyclotomic, so $S$ has geometric Picard rank 2 over $\F_{31}$, and hence geometric Picard rank 1 over $\Q$ by a result of Hassett and V\'arilly-Alvarado \cite[Prop.~5.3]{hva}.  (Note that this reference requires the geometric Picard group of $S_{\F_{31}}$ to have rank 2 \emph{and} be generated by the curves in the preimage of the tritangent line, but the latter is automatic: the intersection pairing between the two curves is
\[ \begin{pmatrix} -2 & 3 \\ 3 & -2 \end{pmatrix}, \]
whose discriminant $-5$ is squarefree, whereas if they generated an index-$N$ sublattice then the discriminant would be divisible by $N^2$.)

\subsection{Computational discussion} \label{computer_chat}

The computational difficulty is in finding all primes $p$ such that there is a tritangent line to the sextic curve defined over $\overline\F_p$.  For any given $p$, we can find tritangent lines defined over $\overline\F_p$ in a fraction of a second using Elsenhans and Jahnel's algorithm \cite[Alg.~8]{ej_deg2}: we write equations in the coefficients of a general line that say it is tritangent to the sextic, and compute a Gr\"obner basis of the ideal they generate.  But if we want to find tritangent lines for all primes at once, we must compute a Gr\"obner basis of the same ideal over $\Z$, which typically takes about half an hour in our implementation.  Moreover, the running time is very sensitive to the details of the implementation: if we make a seemingly trivial change, like switching the order of two variables, it might take hours or days.  In {\tt verify.magma} we carry out the Gr\"obner basis computation over $\Z$, but in our initial search for $f$ we needed something much faster.

Following advice from S.~Elsenhans, we computed a Gr\"obner basis over $\Q$ in Magma with the {\tt ReturnDenominators} option enabled; this returns a list of all the denominators used in the division steps of the Gr\"obner basis algorithm, so if there are no tritangent lines over $\overline\Q$ then the primes at which tritangent lines occur must divide one of those denominators.  In our case the list was very long, and many of the denominators were more than 1000 digits, which is too big to factor.  So we ran the computation twice, with slight variations in the details, and then took common factors between the two lists.  This yielded a list of small numbers and one 300-digit number.  This is still too big to factor in general, but we tested many sextics and occasionally found candidates for which the big number had a few small prime factors and one big prime factor, as in the example above.

We also modified \cite[Alg.~8]{ej_deg2} as follows.  The main step in the algorithm is to take the ideal in
\[ R[a,b,c_0,\dotsc,c_3], \]
where $R = \F_p$ or $\Q$ or $\Z$, generated by equating coefficients in
\[ f(1, t, a+bt) = (c_0 + c_1 t + c_2 t^2 + c_3 t^3)^2, \]
and eliminate the variables $c_0, \dotsc, c_3$ to get a Gr\"obner basis of an ideal in $R[a,b]$ that says ``the line $z = ax + by$ is a tritangent line.''  Over $\F_p$ and $\Q$, this is very efficient.  But over $\Z$ and with {\tt ReturnDenominators} over $\Q$ it was profitable to split it into two steps.  First, take the ideal in
\[ \Z[c_0,\dotsc,c_3,d_0,\dotsc,d_6] \]
generated by equating coefficients in
\[ d_0 + d_1 t + \dotsb + d_6 t^6 = (c_0 + \dotsb + c_3 t^3)^2, \]
and eliminate $c_0, \dotsc, c_3$ to give an ideal in $\Z[d_0,\dotsc,d_6]$ with 51 generators.  This takes a fraction of a second.  Then substitute the coefficients of $f(1,t,a+bt)$ into the latter ideal and compute a Gr\"obner basis.  Our modified algorithm ran at least an order of magnitude faster than the original when using {\tt ReturnDenominators} over $\Q$.  Over $\Z$, it ran in about half an hour, whereas the original ran out of memory before returning an answer.  Our modification produces a much bigger set of generators for the ideal, with elements of much higher degree,
so we were surprised that it performed better.

\newcommand \httpurl [1] {\href{https://#1}{\nolinkurl{#1}}}
\bibliographystyle{amsplain}
\bibliography{rational}

\vspace{2\baselineskip}
\scriptsize
\noindent Nicolas Addington \\
Department of Mathematics \\
University of Oregon \\
Eugene, OR 97403-1222 \\
adding@uoregon.edu \bigskip

\noindent Benjamin Antieau \\
Department of Mathematics\\
Northwestern University\\
2033 Sheridan Road\\
Evanston, IL 60208\\
antieau@northwestern.edu \bigskip

\noindent Katrina Honigs \\
Department of Mathematics \\
University of Oregon \\
Eugene, OR 97403-1222 \\
honigs@uoregon.edu \bigskip

\noindent Sarah Frei \\
Department of Mathematics \\
Rice University \\
6100 Main Street \\
Houston, TX 77005-1892 \\
sarah.frei@rice.edu

\end{document}